\documentclass[a4paper, 12pt]{article}

\newif\ifAMS
\IfFileExists{amssymb.sty}
  {\AMStrue\usepackage{amssymb}}{}
\usepackage{latexsym}

\usepackage{amsthm}
\usepackage{amsmath}

\usepackage{graphicx,textcomp,booktabs,amsmath,pst-node}
\usepackage{mathptmx,courier}
\usepackage[scaled]{helvet}

\newtheorem{thm}{Theorem}[section]
\newtheorem{cor}[thm]{Corollary}
\newtheorem{prop}{Proposition}[section]
\newtheorem{lem}[thm]{Lemma}

\theoremstyle{definition}
\newtheorem{defn}{Definition}[section]

\numberwithin{equation}{section}

\usepackage{amsmath}

\usepackage[ansinew]{inputenc}
\usepackage[T1]{fontenc}
\usepackage{mathptmx,courier}
\usepackage[scaled]{helvet}
\usepackage[arrow, matrix, curve]{xy}
\usepackage{hyperref}

\begin{document}
\bigskip
\bigskip
\bigskip
\bigskip
\bigskip

\title{Tilting and Refined \\ 
Donaldson-Thomas Invariants}

\bigskip
\bigskip
\bigskip
\bigskip

\author{Magnus Engenhorst\footnote{engenhor@math.uni-bonn.de}\\
Mathematical Institute, University of Bonn\\
Endenicher Allee 60, 53115 Bonn, Germany\\[5mm]}

\maketitle

\begin{abstract}
We study tilting for the heart $\mathcal{A}$ of the canonical t-structure of the finite-dimensional derived category of the Ginzburg algebra for a quiver with potential $(Q,W)$. We give conditions on that the stable objects for a central charge on $\mathcal{A}$ define a sequence of simple tilts from $\mathcal{A}$ to $\mathcal{A}[-1]$. On that conditions the refined Donaldson-Thomas invariant associated to $(Q,W)$ is independent of the chosen central charge.
\end{abstract}

\bigskip

\section{Introduction}

For a class of supersymmetric quantum field theories the BPS states are encoded as stable representations of a quiver with potential, see e.g. \cite{1,2}. An algorithm to derive the stable representations without going into linear algebra directly was developed in \cite{50}. This mutation method is based on the idea that the BPS spectrum of a theory is also encoded in the Seiberg dual theory. The mathematically counterpart of Seiberg duality \cite{21} is mutation of quivers with potentials $(Q,W)$ \cite{22,120}. An idea of Bridgeland is that mutation is modeled by tilting of hearts of t-structures of triangulated categories \cite{3}. Inspired from the mutation method in physics we study tilting for the heart of the canonical t-structure of the finite-dimensional derived category of the Ginzburg algebra \cite{90} of $(Q,W)$.\\

In sections 2 and 3 we review Bridgeland stability conditions and tilting theory. In section 4 we consider hearts $\mathcal{A}$ of bounded t-structures of triangulated categories of finite length with finitely many simple objects such that we can tilt indefinitely. We use a version of the mutation method to prove the main result (Prop. 4.1) of this work: A discrete central charge on such heart $\mathcal{A}$ with finitely many stable objects induces a sequence of simple tilts from $\mathcal{A}$ to $\mathcal{A}[-1]$. Conversely, this provides an algorithm to derive the stable objects of $\mathcal{A}$. In section 5 we study examples. An example is $\mathcal{H}_{Q}:=mod-kQ$, the category of representations of an acyclic quiver $Q$ inside the bounded derived category of $\mathcal{H}_{Q}$. If we choose a 2-acyclic quiver $Q$ with non-degenerate potential $W$ in the sense of \cite{80} we can mutate $(Q,W)$ indefinitely which implies we can tilt the heart $\mathcal{A}$ of the canonical t-structure of the finite-dimensional derived category $D_{fd}(\Gamma)$ of the Ginzburg algebra $\Gamma$ of $(Q,W)$ indefinitely. The stable objects with respect to a discrete central charge on $\mathcal{A}$ in the order of decreasing phase define a path in the exchange graph of $D_{fd}(\Gamma)$ (Theorem \ref{theorem}). As a corollary we get that the Jacobi algebra of such $(Q,W)$ is finite-dimensional. In the special case of an acyclic quiver $Q$ we show that the stable objects of $\mathcal{H}_{Q}$ induce a sequence of simple tilts from $\mathcal{A}$ to $\mathcal{A}[-1]$ in $D_{fd}(\Gamma)$ (Corollary 5.4). In section 6 we study the relation to maximal green sequences. The refined Donaldson-Thomas (DT) invariant \cite{150, 160} of a quiver with potential $(Q,W)$ can be read off the mutation method. The theory developed by Kontsevich and Soibelman in \cite{160} suggests it is independent of the chosen central charge.  We prove in section 7 that the refined DT invariants as defined in \cite{140} are equal for discrete central charges with finitely many stable objects (Prop. 7.1).

\section{Tilting}

\begin{defn}\cite{20000}
A \textit{t-structure} on a triangulated category $\mathcal{D}$ is a strictly full subcategory $\mathcal{F}\subset \mathcal{D}$ such that
\begin{enumerate}
\item $\mathcal{F}[1]\subset \mathcal{F}$
\item for every object $E\in T\mathcal{D}$ there is a triangle in $\mathcal{D}$ $$F\longrightarrow E \longrightarrow G\longrightarrow$$
with $F\in\mathcal{F}$ and $G\in\mathcal{F}^{\bot}$ where $$\mathcal{F}^{\bot}=\left\{E\in \mathcal{D}|Hom(A,E)=0, \forall A\in \mathcal{F}\right\}.$$
\end{enumerate}
The \textit{heart of a t-structure} $\mathcal{F}\subset \mathcal{D}$ is the full subcategory $\mathcal{A}=\mathcal{F}\cap\mathcal{F}^{\bot}[1]$.
\end{defn}

A t-structure $\mathcal{F}\subset \mathcal{D}$ is \textit{bounded} if $$\mathcal{D}=\bigcup_{i,j\in\mathbb{Z}}\mathcal{F}[i]\cap \mathcal{F}^{\bot}[j].$$ 

For two hearts $\mathcal{A}_{1},\mathcal{A}_{2}$ with associated bounded t-structures $\mathcal{F}_{1}, \mathcal{F}_{2}\subset \mathcal{D}$ we say $\mathcal{A}_{1}\leq \mathcal{A}_{2}$ if and only if $\mathcal{F}_{2}\subset\mathcal{F}_{1}$.

\begin{lem}\cite{5}
A bounded t-structure is determined by its heart. Moreover, if $\mathcal{A}\subset\mathcal{D}$ is a full additive subcategory of a triangulated category $\mathcal{D}$, then $\mathcal{A}$ is the heart of a bounded t-structure on $\mathcal{D}$ if and only if the following conditions hold:
\begin{enumerate}
\item if $A$ and $B$ are objects of $\mathcal{A}$, then $Hom_{\mathcal{D}}(A,B[k])=0$ for $k<0$,
\item for every nonzero object $E\in\mathcal{D}$ there are integers $m<n$ and a collection of triangles 
$$0=\xymatrixcolsep{4mm}\xymatrix{E_{m}\ar[rr] && E_{m+1}\ar[ld]\ar[rr] && E_{m+2}\ar[ld]\\ \
			& A_{m+1}\ar @{-->}[lu]			& & A_{m+2}\ar @{-->}[lu]}\rightarrow\cdots\rightarrow\xymatrixcolsep{4mm}\xymatrix{ E_{n-1}\ar[rr] && E_{n}\ar[ld]\\			
			 & A_{n}\ar @{-->}[lu]}=E$$
			 with $A_{i}[i]\in\mathcal{A}$ for all $i$.
\end{enumerate} 
\end{lem} 

The objects $A_{i}[i]\in\mathcal{A}$ are called cohomology objects of $E$ with respect to the given t-structure in analogy to the standard t-structure of the derived category of an Abelian category and are denoted $H^{i}(E)$.

\begin{defn}
A \textit{torsion pair} in an Abelian category $\mathcal{A}$ is a pair of full subcategories $(\mathcal{T}, \mathcal{F})$ satisfying
\begin{enumerate}
\item $Hom_{\mathcal{A}}(T,F)=0$ for all $T\in\mathcal{T}$ and $F\in\mathcal{F}$;
\item every object $E\in\mathcal{A}$ fits into a short exact sequence
\begin{align}
0\longrightarrow T\longrightarrow E\longrightarrow F\longrightarrow 0 \nonumber
\end{align}
for some pair of objects $T\in\mathcal{T}$ and $F\in\mathcal{F}$.
\end{enumerate} 
\end{defn}  

The objects of $\mathcal{T}$ are called \textit{torsion} and the objects of $\mathcal{F}$ are called \textit{torsion-free}. We have the following

\begin{prop} ~\cite{20}
Let $\mathcal{A}$ be the heart of a bounded t-structure on a triangulated category $\mathcal{D}$. Let $H^{i}(E)\in\mathcal{A}$ be the i-th cohomology object of E with respect to this t-structure. Let $(\mathcal{T}, \mathcal{F})$ be a torsion pair in $\mathcal{A}$. Then the full subcategory
\begin{align}
\mathcal{A}^{*}&=\left\langle\mathcal{F},\mathcal{T}[-1] \right\rangle \nonumber \\
&=\left\{E\in \mathcal{D}| H^{i}(E)=0 \text{ for } i\notin \lbrace0,1\rbrace, H^{0}(E)\in\mathcal{F}, H^{1}(E)\in \mathcal{T}\right\} \nonumber
\end{align}
is the heart of a bounded t-structure on $\mathcal{D}$. 
\end{prop}
We say $\mathcal{A}^{*}$ is obtained from $\mathcal{A}$ by \textit{(left) tilting} with respect to the torsion pair $(\mathcal{T}, \mathcal{F})$.\\ 

The following lemma of Bridgeland gives a composition of left tilts.

\begin{lem}\cite{190}
\label{nagaobridgeland}
Let $(\mathcal{T}, \mathcal{F})$ be a torsion pair in $\mathcal{A}$ and $(\mathcal{T}', \mathcal{F}')$ a torsion pair in $\mathcal{A}^{*}=\left\langle\mathcal{F},\mathcal{T}[-1] \right\rangle$. If $\mathcal{T}'\subset \mathcal{F}$, then the left tilt $\mathcal{A}^{**}=\left\langle\mathcal{F}',\mathcal{T}'[-1] \right\rangle$ of $\mathcal{A}^{*}$ equals a left tilt of $\mathcal{A}$.  
\end{lem}

Note that if $\mathcal{A}^{*}=\left\langle\mathcal{F},\mathcal{T}[-1] \right\rangle$ is the left tilt of $\mathcal{A}$ with respect to a torsion pair $(\mathcal{T}, \mathcal{F})$, then the simple objects of $\mathcal{A}^{*}$ lie in $\mathcal{A}$ or in $\mathcal{A}[-1]$: We have a short exact sequence in $\mathcal{A}^{*}$ for every object $S\in\mathcal{A}^{*}$ $$0\longrightarrow E\longrightarrow S\longrightarrow F\longrightarrow 0$$ with $E\in\mathcal{F}\subset \mathcal{A}$ and $F\in\mathcal{T}[-1]\subset\mathcal{A}[-1]$. If $S$ is simple we have $S\cong E$ or $S\cong F$.\\ 

Suppose $\mathcal{A}\subset\mathcal{D}$ is the heart of a bounded t-structure on $\mathcal{D}$ and of finite length. Given a simple $S\in\mathcal{A}$ we denote by $\left\langle S\right\rangle$ the full subcategory of objects $E\in\mathcal{A}$ whose simple factors are isomorphic to $S$. We can view $\left\langle S\right\rangle$ as the torsion part of a torsion pair on $\mathcal{A}$ with torsion-free part $$\mathcal{F}=\left\{E\in\mathcal{A}|Hom_{\mathcal{A}}(S,E)=0\right\}$$ or as a torsion-free part with torsion part $$\mathcal{T}=\left\{E\in\mathcal{A}|Hom_{\mathcal{A}}(E,S)=0\right\}.$$

The new hearts after tilting are
\begin{align}
L_{S}(\mathcal{A})&=\left\{E\in \mathcal{D}| H^{i}(E)=0 \text{ for } i\notin \lbrace0,1\rbrace, H^{0}(E)\in\mathcal{F}, H^{1}(E)\in \left\langle S\right\rangle\right\}, \nonumber \\
R_{S}(\mathcal{A})&=\left\{E\in \mathcal{D}| H^{i}(E)=0 \text{ for } i\notin \lbrace-1,0\rbrace, H^{-1}(E)\in\left\langle S\right\rangle, H^{0}(E)\in\mathcal{T} \right\}. \nonumber
\end{align}
$L_{S}(\mathcal{A})$ (respectively $R_{S}(\mathcal{A})$) is called \textit{the left} (respectively \textit{the right}) \textit{tilt of $\mathcal{A}$ at the simple $S$}. $S[-1]$ is a simple object in $L_{S}(\mathcal{A})$ and if this heart is again of finite length we have $R_{S[-1]}L_{S}(\mathcal{A})=\mathcal{A}$. Similarly, if $R_{S}(\mathcal{A})$ has finite length, we have $L_{S[1]}R_{S}(\mathcal{A})=\mathcal{A}$.

\section{Stability conditions on triangulated categories}

We review stability conditions on a triangulated category $\mathcal{D}$ introduced by Bridgeland in \cite{5}. We denote by $K(\mathcal{D})$ the corresponding Grothendieck group of $\mathcal{D}$.

\begin{defn}~\cite{5}
\label{bridgeland}
A \textit{stability condition} on a triangulated category $\mathcal{D}$ consists of a group homomorphism $Z:K(\mathcal{D})\rightarrow \mathbb{C}$ called the \textit{central charge} and of full additive subcategories $\mathcal{P}(\phi)\subset\mathcal{D}$ for each $\phi\in \mathbb{R}$, satisfying the following axioms:
\begin{enumerate}
\item if $0\neq E\in\mathcal{P}(\phi)$, then $Z(E)=m(E)exp(i\pi\phi)$ for some $m(E)\in \mathbb{R}_{>0}$;
\item $\forall \phi\in\mathbb{R}, \mathcal{P}(\phi+1)=\mathcal{P}(\phi)\left[1\right]$;
\item if $\phi_{1}>\phi_{2}$ and $A_{j}\in\mathcal{P}(\phi_{j})$, then $Hom_{\mathcal{D}}(A_{1},A_{2})=0;$
\item for $0\neq E\in\mathcal{D}$, there is a finite sequence of real numbers $\phi_{1}>\cdots>\phi_{n}$ and a collection of triangles 

$$0=\xymatrixcolsep{7mm}\xymatrix{E_{0}\ar[rr] && E_{1}\ar[ld]\ar[rr] && E_{2}\ar[ld]\\ \
			& A_{1}\ar @{-->}[lu]			& & A_{2}\ar @{-->}[lu]}\rightarrow\cdots\rightarrow\xymatrixcolsep{7mm}\xymatrix{ E_{n-1}\ar[rr] && E_{n}\ar[ld]\\			
			 & A_{n}\ar @{-->}[lu]}=E$$
       
with $A_{j}\in\mathcal{P}(\phi_{j})$ for all j. 
\end{enumerate}
\end{defn} 

We recall some results of \cite{5}. The subcategory $\mathcal{P}(\phi)$ is Abelian and its nonzero objects are said to be \textit{semistable} of phase $\phi$ for a stability condition $\sigma=(Z,\mathcal{P})$. We call its simple objects \textit{stable}. The objects $A_{i}$ in Definition $\ref{bridgeland}$ are called semistable factors of E with respect to $\sigma$. For any interval $I\subset\mathbb{R}$ we define $\mathcal{P}(I)$ to be the extension-closed subcategory of $\mathcal{D}$ generated by the subcategories $\mathcal{P}(\phi)$ for $\phi\in I$.\\

A stability condition is \textit{locally-finite} if there exists some $\epsilon>0$ such that for all $\phi\in\mathbb{R}$ each quasi-Abelian subcategory $\mathcal{P}((\phi-\epsilon,\phi+\epsilon))$ is of finite length. We denote by $Stab(\mathcal{D})$ the set of locally-finite stability conditions. It is a topological space.\\

A \textit{central charge} (or \textit{stability function}) on an Abelian category $\mathcal{A}$ is a group homomorphism $Z:K(\mathcal{A})\rightarrow\mathbb{C}$ such that for any nonzero $E\in\mathcal{A}$, $Z(E)$ lies in the upper halfplane
\begin{align}
\label{halfplane}
\mathbb{H}:=\left\{r\cdot exp(i\pi\phi)| 0<\phi\leq 1,r\in\mathbb{R}_{>0}\right\}\subset\mathbb{C}.
\end{align}

Every object $E\in\mathcal{A}$ has a phase $0<\phi(E)\leq 1$ such that $Z(E)=r\cdot exp(i\pi\phi(E))$ with $r\in\mathbb{R}_{>0}$. We say a nonzero object $E\in\mathcal{A}$ is \textit{semistable} (respectively \textit{stable}) with respect to the central charge $Z$ if every subobject $0\neq A\subset E$ satisfies $\phi(A)\leq \phi(E)$ (respectively $\phi(A)< \phi(E)$). The central charge $Z$ has the Harder-Narasimhan property if every nonzero object $E\in \mathcal{A}$ has a finite filtration $$0=E_{0}\subset E_{1}\subset \ldots \subset E_{n-1}\subset E_{n}=E$$ where the semistable factors $F_{j}=E_{j}/E_{j-1}$ fulfill $$\phi(F_{1})>\phi(F_{2})>\ldots >\phi(F_{n}).$$
 
\begin{prop}~\cite{5}
To give a stability condition on a triangulated category $\mathcal{D}$ is equivalent to giving a bounded t-structure on $\mathcal{D}$ and a central charge on its heart which has the Harder-Narasimhan property.
\end{prop}
\begin{proof}Given a heart $\mathcal{A}$ of a bounded t-structure on $\mathcal{D}$ and a central charge with HN property we define the subcategories $\mathcal{P}(\phi)$ to be the semistable objects of $\mathcal{A}$ of phase $\phi\in(0,1]$ and continue by the rule $\mathcal{P}(\phi+1)=\mathcal{P}(\phi)[1]$.\\ Conversely, given a stability condition $\sigma=(Z,\mathcal{P})$ on a triangulated category $\mathcal{D}$ the full subcategory $\mathcal{A}=\mathcal{P}((0,1])$ is the heart of a bounded t-structure on $\mathcal{D}$. Identifying the Grothendieck groups $K(\mathcal{A})$ and $K(\mathcal{D})$ the central charge $Z:K(\mathcal{D})\rightarrow\mathbb{C}$ defines a central charge on $\mathcal{A}$. The semistable objects of the categories $\mathcal{P}(\phi)$ are the semistable objects of $\mathcal{A}$ with respect to this central charge.
\end{proof}   

Let $\mathcal{A}\subset\mathcal{D}$ be the heart of a bounded t-structure on a triangulated category $\mathcal{D}$. We further assume that $\mathcal{A}$ is of finite length with finitely many simple objects $S_{1},\ldots,S_{n}$. Then the subset $U(\mathcal{A})$ of $Stab(\mathcal{D})$ consisting of stability conditions with heart $\mathcal{A}$ is isomorphic to $\mathbb{H}^{n}$.\\

We are interested in the case of a simple object of $\mathcal{A}$ leaving the upper halfplane. We have the following crucial result:

\begin{prop}
\label{lemma} 
(\cite{10}, Lemma 5.5) Let $\mathcal{A}\subset \mathcal{D}$ be the heart of a bounded t-structure on $\mathcal{D}$ and suppose $\mathcal{A}$ has finite length with finitely many simple objects. Then the codimension one subset of $U(\mathcal{A})$ where the simple $S$ has phase $1$ and all other simples in $\mathcal{A}$ have phases in $(0,1)$ is the intersection $U(\mathcal{A})\cap \overline{U(\mathcal{B})}$ precisely if $\mathcal{B}=L_{S}(\mathcal{A})$.                   
\end{prop}

The tilted subcategory $L_{S}(\mathcal{A})$ need not to have finite length. If we can tilt at a simple of the new heart again the corresponding regions in $Stab(\mathcal{D})$ can be glued together at their codimension one boundaries and so on.

\section{Mutation method}

In this section $\mathcal{A}\subset\mathcal{D}$ is the heart of a bounded t-structure of $\mathcal{D}$ such that we can tilt indefinitely, i.e. $\mathcal{A}$ is of finite length with finitely many simple objects and every heart obtained from $\mathcal{A}$ by a sequence of simple tilts is again of finite length. We will show that a discrete central charge induces a sequence of tilts from the heart $\mathcal{A}$ to the heart $\mathcal{A}[-1]$. 

\begin{defn} We call a central charge $Z:K(\mathcal{A})\rightarrow\mathbb{C}$ \textit{discrete} if two stable objects of $\mathcal{A}$ have the same phase precisely if they are isomorphic. 
\end{defn}

Let us consider the $n$ simple objects $S_{1},\ldots, S_{n}$ of $\mathcal{A}$. For a discrete central charge there must be a simple $S_{i}$ that is left-most, i.e. whose phase is the biggest. We identify the Grothendieck-groups $K(\mathcal{A})=K(\mathcal{D})$ and rotate the complex numbers $Z(S_{1}),\ldots, Z(S_{n})$ a bit counterclockwise, such that the left-most simple $S_{i}$ just leaves the upper halfplane. Then we tilt at this simple. Prop. \ref{lemma} tells us that the corresponding stability condition in $U(\mathcal{A})$ crosses the boundary of $U(L_{S_{i}}(\mathcal{A}))$ and we end up with a stability condition in $U(L_{S_{i}}(\mathcal{A}))$. Then we rotate further and proceed with this procedure until (if possible) we accomplish a rotation by $\pi$. This algorithm describes a path through the space of stability conditions. The stable objects do not change during rotation as long as all simples stay in the upper halfplane. This procedure is inspired by the mutation method in \cite{50}.\\

In the proof of Prop. \ref{lemma} we use a version of 

\begin{lem}\label{primitive} Let $\mathcal{D}$ be a triangulated category such that $K(\mathcal{D})$ is a finite-dimensional lattice and let $Stab^{*}(\mathcal{D})\subset Stab(\mathcal{D})$ be a full connected component. Let us assume that the object $E\in\mathcal{D}$ has primitive class in $K(\mathcal{D})$. If $E$ is stable in a stability condition $\sigma\in Stab^{*}(\mathcal{D})$, then it is stable in a neighborhood of $\sigma$.                   
\end{lem}   
\begin{proof}
This follows from the arguments of \cite{30}, section 9, see \cite{40}.  
\end{proof}

This means if we cross the real line the simple objects of $\mathcal{A}$ will remain stable in $U(L_{S_{i}}(\mathcal{A}))$ near to the boundary. A priori, two simple objects of a tilted heart could be both left-most. We exclude this possibility in Lemma \ref{lemma2} and we can therefore continue indefinitely with the mutation algorithm described above. Let us assume we rotate by finitely many steps in the mutation method. Then this is the key result:

\begin{prop}
\label{prop}
Let $\mathcal{A}\subset\mathcal{D}$ be the heart of a bounded t-structure of $\mathcal{D}$ with discrete central charge $Z:K(\mathcal{A})\rightarrow\mathbb{C}$ as described above.  The left-most simple objects of hearts appearing in the mutation method are the stable objects of $\mathcal{A}$. In the order of decreasing phase they give a sequence of simple tilts from $\mathcal{A}$ to $\mathcal{A}[-1]$. In particular, we tilt at all initial simple objects $S_{1},\ldots, S_{n}$.
\end{prop}
\begin{proof}
In the mutation method we always tilt at objects in $\mathcal{A}$: The first tilt is at a simple object in $\mathcal{A}$. Then the simple objects in the first tilted heart are in $\mathcal{A}$ or in $\mathcal{A}[-1]$. Since we tilt at the left-most object this implies we tilt at an object in $\mathcal{A}$. It follows from lemma \ref{nagaobridgeland} by induction that the simple objects of a tilted heart are in $\mathcal{A}$ or in $\mathcal{A}[-1]$. The final heart $\mathcal{A}'$ obtained in the mutation method contains only simple objects in $\mathcal{A}[-1]$ since we have then rotated by $\pi$. We have therefore  $\mathcal{A}'\subset\mathcal{A}[-1]$ and this implies $\mathcal{A}'=\mathcal{A}[-1]$. If simple objects of a tilted heart are in $\mathcal{A}$ one of these is left-most and we tilt at it. If all simple objects are in $\mathcal{A}[-1]$ we are in the final heart.\\

Lemma \ref{primitive} implies that all objects appearing as left-most simple objects of some heart appearing during this procedure are stable objects in $\mathcal{A}$. The phases of all stable objects in $\mathcal{A}$ are smaller than the phase of the left-most simple object $S$. By the definition of the left-tilt all stable objects except the left-most simple remain in the first tilt of $\mathcal{A}$ since there are no homomorphisms between $S$ and the other stable objects. In the first tilted heart the phases of the stable objects of $\mathcal{A}$ are equal or smaller than the new left-most simple object. If the phase of a stable object of $\mathcal{A}$ is equal to this left-most simple they are the same since we chose a discrete central charge. Otherwise the stable object remains in the next tilted heart and so on. Therefore we tilt in the mutation method at all stable objects of $\mathcal{A}$. For every central charge, we tilt at all inital simple objects $S_{1},\ldots, S_{n}$.  
\end{proof}

\begin{cor}
For every heart $\mathcal{A}'$ appearing in the mutation method we have $\mathcal{A}[-1]\leq\mathcal{A}'\leq\mathcal{A}$.
\end{cor}

\begin{lem}\label{lemma2}
Let $\mathcal{A}\subset\mathcal{D}$ be the heart of a bounded t-structure of $\mathcal{D}$ with discrete central charge $Z:K(\mathcal{A})\rightarrow\mathbb{C}$ as described above. The phases of any simple objects of a heart in any step of the mutation method are distinct.      
\end{lem}
\begin{proof}
We saw in section 2 that right tilting is inverse to left tilting and vice versa. Instead of the upper halfplane $\mathbb{H}\cup \mathbb{R}_{<0}$ we could have chosen the convention $\mathbb{H}\cup \mathbb{R}_{>0}$ in the definition of a central charge. So the objects with primitive class in $K(\mathcal{D})$ remain stable along the path described above in both directions. If two simple objects in a heart had the same phase this would mean there are two stable objects in the initial heart $\mathcal{A}$ with the same phase. Since we chose a discrete central charge, this is a contradiction.   
\end{proof}

Note if there are only finitely many stable objects in $\mathcal{A}$ for a discrete central charge the proof of Lemma \ref{lemma2} implies we rotate by finitely many steps. Indeed, in this case we have only a finite set of objects that can appear as simple objects in a tilted heart. Since there are no oriented cycles in the exchange graph of directed simple left-tilts the mutation method must terminate after finitely many steps.

\section{Quivers with (super)potential}

Let $k$ be a field. In this section we consider examples of hearts of bounded t-structures of triangulated categories such that we can tilt indefinitely. The first example is the category of representations $\mathcal{H}_{Q}:=mod-kQ$ of an acyclic quiver $Q$. $\mathcal{H}_{Q}$ is the heart of the standard t-structure on the derived category of $\mathcal{H}_{Q}$. By theorem 5.7 in \cite{70} every heart obtained from $\mathcal{H}_{Q}$ by a sequence of simple tilts is of finite length with finitely many simple objects. In the special case of a Dynkin quiver with discrete central charge Prop. \ref{prop} reads as follows:

\begin{prop}\cite{60}
Let $\mathcal{H}_{Q}$ be the category of representations of a Dynkin quiver $Q$. Then the stable representations of $\mathcal{H}_{Q}$ in the order of decreasing phase give a sequence of simple tilts from $\mathcal{H}_{Q}$ to $\mathcal{H}_{Q}[-1]$. 
\end{prop}

An example for a non-Dynkin quiver is the Kronecker quiver 
  $$\xymatrix{1 \ar@<+.7ex>[r]\ar@<-.7ex>[r] & 2}.  \\$$

Let us denote by $S_{1}$ and $S_{2}$ the simple representations associated with the two vertices. If the phase of $S_{2}$ is strictly greater than the phase of $S_{1}$, the simples are the only stable objects and we tilt two times to get to the heart with simples $S_{1}[-1], S_{2}[-1]$. If the phase of $S_{1}$ is strictly greater than the phase of $S_{2}$ the stable objects are precisely the representations in the $\mathbb{P}^{1}$-family with dimension vector $(1,1)$ together with the postprojective and the preinjective representations. In this case infinitely many stable objects lie on a ray in the upper halfplane.\\

In general, we can order the simple objects $S_{1},\ldots, S_{n}$ of an acyclic quiver with $n$ vertices so that $$Ext^{1}(S_{j},S_{i})=0\text{ for } 1\leq i< j\leq n.$$ By the proof of Proposition \ref{prop} we can find for any acyclic quiver a discrete central charge such that the stable objects are precisely the simple objects.\\

The mutation method in \cite{50} uses mutations of quivers with potential. An idea of Bridgeland was that mutation is modeled by tilting hearts \cite{1}. This philosophy is behind Theorem \ref{kelleryang}. We now make contact with these original ideas. 

\begin{defn}
Let $Q$ be a finite, 2-acyclic\footnote{We call a quiver $2-acyclic$ if it does not contain loops $\circlearrowleft$ or 2-cycles $\leftrightarrows$.} quiver and $r$ be a vertex of $Q$. The \textit{mutation} of $Q$ at the vertex $r$ is the new quiver $\mu_{r}(Q)$ obtained from $Q$ by the rules:
\begin{enumerate}
\item for each $i\rightarrow r \rightarrow j$ add an arrow $i\rightarrow j$,
\item reverse all arrows with source or target $r$,
\item remove a maximal set of 2-cycles.
\end{enumerate} 
\end{defn}  

In the following example
       
         $$
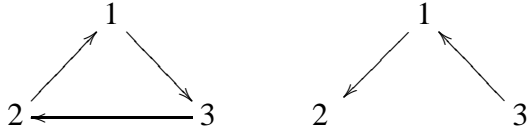
\xymatrix{& 1 \ar[rd]\\
       2\ar[ru] && 3\ar[ll]} \hspace{2.5pc} \xymatrix{& 1 \ar[ld]\\
       2\ && 3\ar[lu]} \hspace{2.5pc}$$
       
the quivers are linked by a mutation at the vertex $1$.\\ 

The category of representations of an acyclic quiver is a special case of the category of finite-dimensional modules over the Jacobi algebra of a quiver with potential \cite{80}. Let $Q=(Q_{0},Q_{1})$ be a finite quiver with set of vertices $Q_{0}$ and set of arrows $Q_{1}$. We denote by $kQ$ its path algebra, i.e. the algebra with basis given by all paths in $Q$ and product given by composition of paths. Let $\widehat{kQ}$ be the completion of $kQ$ at the ideal generated by the arrows of $Q$. We consider the quotient of $\widehat{kQ}$ by the subspace $[\widehat{kQ},\widehat{kQ}]$ of all commutators. It has a basis given by the cyclic paths of $Q$ (up to cyclic permutation). For each arrow $a\in Q_{1}$ the cyclic derivative is the linear map from the quotient to $\widehat{kQ}$ which takes an equivalence class of a path $p$ to the sum $$\sum_{p=uav}vu$$ taken over all decompositions $p=uav$. An element $$W\in \widehat{\frac{\widehat{kQ}}{[\widehat{kQ},\widehat{kQ}]}}$$ is called a \textit{(super)potential} if it does not involve cycles of length $\leq 2$.

\begin{defn}\cite{80}
Let $(Q,W)$ be a quiver $Q$ with potential $W$. The \textit{Jacobi algebra} $\mathfrak{P}(Q,W)$ is the quotient of $\widehat{kQ}$ by the two-sided ideal generated by the cyclic derivatives $\partial_{a}W$: $$\mathfrak{P}(Q,W):=\widehat{kQ}/(\partial_{a}W, a\in Q_{1}).$$
\end{defn} 

We call a quiver with potential $(Q,W)$ \textit{Jacobi-finite} if its Jacobi algebra is finite-dimensional. We denote by $nil(\mathfrak{P}(Q,W))$ the category of finite-dimensional (right) modules over $\mathfrak{P}(Q,W)$. This is an Abelian category of finite length with simple objects the modules $S_{i}, i\in Q_{1}$. Given a quiver with potential we introduce a triangulated category following \cite{100}. This category has a canonical t-structure with heart equivalent to $nil(\mathfrak{P}(Q,W))$.\\

Let $(Q,W)$ be a quiver $Q$ with potential $W$. The \textit{Ginzburg algebra} $\Gamma(Q,W)$ \cite{90} of $(Q,W)$ is the differential graded (dg) algebra constructed as follows: Let $\tilde{Q}$ be the graded quiver\footnote{A graded quiver is a quiver where each arrow is equipped with an integer degree.} with the same vertices as $Q$ and whose arrows are
\begin{enumerate}
\item the arrows of $Q$ (they all have degree 0),
\item an arrow $a^{*}:j\rightarrow i$ of degree $-1$ for each arrow $a:i\rightarrow j$ of $Q$,
\item a loop $t_{i}:i\rightarrow i$ of degree $-2$ for each vertex $i\in Q_{0}$.
\end{enumerate}  

The underlying graded algebra of the Ginzburg algebra $\Gamma:=\Gamma(Q,W)$ is the completion of the graded path algebra $k\tilde{Q}$ in the category of graded vector spaces with respect to the ideal generated by the arrows of $\tilde{Q}$. The differential of $\Gamma(Q,W)$ is the unique continuous linear endomorphism homogeneous of degree 1 which satisfies the Leibniz rule $$d(uv)=(du)v+(-1)^{p}udv,$$ for all homogeneous $u$ of degree $p$ and all $v$ defined by
\begin{enumerate}
\item $da=0$ for each arrow $a$ of $Q$,
\item $d(a^{*})=\partial_{a}W$ for each arrow $a$ of $Q$,
\item $d(t_{i})=e_{i}(\sum_{a}[a,a^{*}])e_{i}$ for each vertex $i$ of $Q$ where $e_{i}$ is the lazy path at $i$.
\end{enumerate}

The Ginzburg algebra is concentrated in cohomological degrees $\leq 0$ and $H^{0}(\Gamma)$ is isomorphic to $\mathfrak{P}(Q,W)$. Let $D(\Gamma)$ be the derived category of the Ginzburg algebra and $D_{fd}(\Gamma)$ be the full subcategory of $D(\Gamma)$ formed by dg modules whose homology is of finite total dimension. For derived categories of differential graded categories see e.g. \cite{100}. The category $D_{fd}(\Gamma)$ is triangulated and 3-Calabi-Yau \cite{110}. Since $\Gamma$ is concentrated in degrees $\leq 0$ the category $D(\Gamma)$ admits a canonical t-structure whose truncation functors are those of the canonical t-structure on the category of complexes of vector spaces \cite{120}. The heart $\mathcal{A}$ of the induced t-structure on $D_{fd}(\Gamma)$ is equivalent to $nil(\mathfrak{P}(Q,W))$. The simple $\mathfrak{P}(Q,W)$-modules $S_{i}$ associated with the vertices of $Q$ are made into $\Gamma$-modules via the morphism $\Gamma\rightarrow H^{0}(\Gamma)$. In $D_{fd}(\Gamma)$ they are 3-spherical objects, i.e. we have an isomorphism $$Ext_{\Gamma}^{*}(S_{i},S_{i})\cong H^{*}(S^{3},\mathbb{C}).$$ For spherical objects in triangulated categories see \cite{125}.\\

The mutation of a 2-acyclic quiver $Q\mapsto \mu_{r}(Q)$ at a vertex $r$ admits a good extension to quivers with potential if the potential $W$ is non-degenerate \cite{80}. We denote the mutation of $(Q,W)$ at the vertex $r$ by $\mu_{r}(Q,W)$. Let $\Gamma$ be the Ginzburg algebra of $(Q,W)$ and $\Gamma'$ be the Ginzburg algebra of $\mu_{r}(Q,W)$.

\begin{thm}\cite{120}
\label{kelleryang}
Let $(Q,W)$ be a 2-acyclic quiver with non-degenerate potential. Then there are two canonical equivalences 
$$\Phi_{\pm}:D(\Gamma')\longrightarrow D(\Gamma)$$ inducing equivalences of the subcategories $$D_{fd}(\Gamma')\longrightarrow D_{fd}(\Gamma).$$ Let $\mathcal{A}'$ be the heart of the canonical t-structure on $D_{fd}(\Gamma')$. Then the equivalences $\Phi_{\pm}$ send $\mathcal{A}'$ to the hearts of two new t-structures on $D_{fd}(\Gamma)$ given by the left respectively right tilt of $\mathcal{A}$ in the sense of section 2.
\end{thm}

The important point for us is the following: The simple objects of $\mathcal{A}$ can be identified with the simple objects $S_{1},\ldots, S_{n}$ of $nil(\mathfrak{P}(Q,W))$ for a quiver $Q$ with $n$ vertices. They generate the heart $\mathcal{A}$. Let $(Q,W)$ be a 2-acyclic quiver $Q$ with a non-degenerate potential $W$ in the sense of \cite{80}. Theorem \ref{kelleryang} implies we can tilt indefinitely at simple objects so that we can apply Prop. \ref{prop}.\\
    
\begin{thm}
\label{theorem}
Let $(Q,W)$ be a 2-acyclic quiver $Q$ with non-degenerate potential $W$ such that we have a discrete central charge on the heart $\mathcal{A}$ of the canonical t-structure of $D_{fd}(\Gamma)$ with finitely many stable objects. Then the sequence of stable objects of $\mathcal{A}$ in the order of decreasing phase defines a sequence of simple tilts from $\mathcal{A}$ to $\mathcal{A}[-1]$. Moreover, $(Q,W)$ is Jacobi-finite.  
\end{thm}
\begin{proof}
We only have to prove the last statement. But this is an immediate consequence of Prop. 8.1 in \cite{155}.
\end{proof}

The dimensions of $Ext^{1}$-groups between the simple objects $S_{1},\ldots, S_{n}$ are given by the quiver $Q$:
\begin{align}
\text{dim } Ext^{*}(S_{j},S_{i})=\# (\text{arrows } i\longrightarrow j \text{ in Q}). \nonumber
\end{align}

Together with Theorem \ref{kelleryang} follows  
 
\begin{lem} 
The sequence of mutations of $Q$ modeled by the sequence of simple tilts in the mutation method of section 4 linking the set $(S_{1},\ldots, S_{n})$ to the set $(S_{1}[-1],\ldots, S_{n}[-1])$ gives back the original quiver $Q$ (up to permutation of the vertices).   
\end{lem}

Let us consider an acyclic quiver $Q$. The category $\mathcal{H}_{Q}:=mod-kQ$ is the heart of the canonical t-structure of the bounded derived category of $\mathcal{H}_{Q}$. $\mathcal{H}_{Q}$ is equivalent to the heart $\mathcal{A}$ of the canonical t-structure of $D_{fd}(\Gamma)$ for the Ginzburg algebra $\Gamma$ of $Q$. We assume we have finitely many stable objects (or equivalently we rotate by finitely many steps in the mutation method).

\begin{cor}
The stable objects of $\mathcal{H}_{Q}$ in the order of decreasing phase induce a sequence of simple tilts from $\mathcal{A}$ to $\mathcal{A}[-1]$ and the stable objects of $\mathcal{A}$ in the order of decreasing phase induce a sequence of simple tilts from $\mathcal{H}_{Q}$ to $\mathcal{H}_{Q}[-1]$.
\end{cor}

An isomorphism of entire exchange graphs for the two derived categories associated to an acyclic quiver is constructed in \cite{70}.

\section{Maximal green sequences}

In this section we relate stable objects to maximal green mutation sequences as introduced by B. Keller in \cite{140}.\\

Let us consider a 2-acyclic quiver $Q$ with $n$ vertices. Let $\tilde{Q}$ be the \textit{principal extension of $Q$}, i.e. the quiver obtained from $Q$ by adding a new vertex $i':=i+n$ and a new arrow $i\rightarrow i'$ for each vertex $i\in Q_{0}$. The new vertices $i'$ are called frozen and we will never mutate at them. Here is an example:
\begin{align}
\label{a2greenred}
Q:\hspace{.5pc}\xymatrix{1 \ar[r]& 2}, \hspace{2.5pc}\widetilde{Q}:\hspace{.5pc}\xymatrix{1\ar[r]\ar[d] & 2\ar[d]\\
       1'& 2'}
\end{align}

\begin{defn}\cite{140}
A vertex $j$ of a quiver in the mutation class of $\widetilde{Q}$ is called \textit{green} if there are no arrows from a frozen vertex $i'$ to $j$ and \textit{red} otherwise. A \textit{green (mutation) sequence} on $\widetilde{Q}$ is a mutation sequence such that every mutation in the sequence is at a green vertex in the corresponding quiver. A green sequence is \textit{maximal} if all vertices of the final quiver are red. The \textit{length} of a green mutation sequence is the number of mutations in the sequence.
\end{defn} 

If in the example above we begin mutating at vertex $2$ we find the maximal green sequence
$$
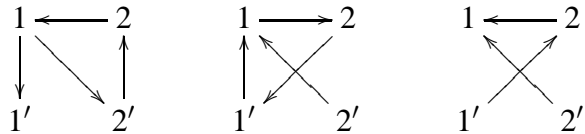
\xymatrix{1\ar[d]\ar[rd]& 2\ar[l]\\
       1'& 2'\ar[u]} \hspace{2.5pc} \xymatrix{1\ar[r]& 2\ar[ld]\\
       1'\ar[u]& 2'\ar[lu]}\hspace{2.5pc} \xymatrix{1& 2\ar[l]\\
       1'\ar[ru]& 2'\ar[lu]}$$
Starting at vertex $1$ we find
$$
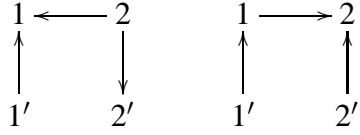
\xymatrix{1& 2\ar[l]\ar[d]\\
       1'\ar[u]& 2'} \hspace{2.5pc} \xymatrix{1\ar[r]& 2\\
       1'\ar[u]& 2'\ar[u]}$$

Let us consider a sequence of simple tilts as in Theorem \ref{theorem}. They define a sequence of nearby cluster collections (see section 7) and give therefore a sequence of green mutations as explained in \cite{140, 100}.

\begin{prop}
\label{prop2}
Let $(Q,W)$ be a quiver with potential as in Theorem \ref{theorem}. Then the stable objects of $\mathcal{A}$ define a maximal green mutation sequence of $\widetilde{Q}$ with length given by the number of stable objects.
\end{prop}    

Let $Q$ be an acyclic quiver. Since we can find a central charge such that the stable objects are exactly the simple objects of $\mathcal{H}_{Q}$, the set of maximal green mutations of $Q$ is non-empty. If $Q$ is a Dynkin quiver there is a discrete central charge with stable objects given by all indecomposable objects. Therefore we can find a maximal green sequence of length equal to the number of indecomposables.

\section{Refined Donaldson-Thomas invariants}

We can associate to a quiver with potential a refined Donaldson-Thomas invariant \cite{150, 160}. In this section we choose $k=\mathbb{C}$ and we closely follow \cite{140}.\\

Let $Q$ be a finite quiver with $n$ vertices. The quantum affine space $\mathbb{A}_{Q}$ is the $\mathbb{Q}(q^{1/2})$-algebra generated by the variables $y^{\alpha},\alpha\in\mathbb{N}^{n}$, subject to the relations $$y^{\alpha}y^{\beta}=q^{1/2\lambda(\alpha,\beta)}y^{\alpha+\beta}$$ where $\lambda(\ ,\ )$ is the antisymmetrization of the Euler form of $Q$. Equivalently, $\mathbb{A}_{Q}$ is generated by the variables $y_{i}:=y^{e_{i}}, 1\leq i\leq n$ subject to the relations $$y_{i}y_{j}=q^{\lambda(e_{i},e_{j})}y_{j}y_{i}.$$ We denote by $\hat{\mathbb{A}}_{Q}$ the completion of $\mathbb{A}_{Q}$ with respect to the ideal generated by the $y_{i}$.\\  

Let $(Q,W)$ be a quiver with potential and we assume we can find a discrete central charge $Z$ on $nil(\mathfrak{P}(Q,W))$. The refined Donaldson-Thomas invariant is defined to be the product in $\hat{\mathbb{A}}_{Q}$ $$\mathbb{E}_{Q,W,Z}:=\vec{\prod}_{M\ stable}\mathbb{E}(y^{\underline{dim}\ M})$$ where the stable modules with respect to the discrete central charge appear in the order of decreasing phase. $\mathbb{E}(y)$ is the quantum dilogarithm \cite{175}, i.e. the element in the power series algebra $\mathbb{Q}(q^{1/2})[[y]]$ defined by
\begin{align}
\mathbb{E}(y)=1+\frac{q^{1/2}}{q-1}y+\ldots + \frac{q^{n^{2}/2}}{(q^{n}-1)(q^{n}-q)\cdots (q^{n}-q^{n-1})}y^{n}+\ldots. \nonumber
\end{align}

The invariant $\mathbb{E}_{Q,W,Z}$ is of course only well defined if it does not depend on the choice of a discrete central charge $Z$. (This is conjecture 3.2 in \cite{140}.) If it is well-defined we denote it by $\mathbb{E}_{Q,W}$.\\

The set of simple objects ($S_{1},\ldots, S_{n}$) of the heart $\mathcal{A}$ of the canonical t-structure of $D_{fd}(\Gamma)$ is a cluster collection.
 
\begin{defn}\cite{160}
A \textit{cluster collection} $S'$ is a sequence of objects $S'_{1},\ldots, S'_{n}$ of $D_{fd}(\Gamma)$ such that

\begin{enumerate}
\item the $S'_{i}$ are spherical,
\item $Ext^{*}(S'_{i},S'_{j})$ vanishes or is concentrated either in degree 1 or degree 2 for $i\neq j$,
\item the $S'_{i}$ generate the triangulated category $D_{fd}(\Gamma)$.
\end{enumerate}
\end{defn}

In our case the cluster collection $S_{1},\ldots, S_{n}$ is linked to the cluster collection $S_{1}[-1],\ldots, S_{n}[-1]$ by a sequence of simple tilts and permutations. 
The functor $[-1]$ is therefore a \textit{reachable} functor for $D_{fd}(\Gamma)$ in the sense of \cite{140, 180}. A functor $F:D_{fd}(\Gamma)\rightarrow D_{fd}(\Gamma)$ is reachable if there is a sequence of mutations and permutations from the initial cluster collection $(S_{1},\ldots, S_{n})$ to $(F(S_{1}),\ldots, F(S_{n}))$.\\

A quiver $Q$ has a associated braid group $Braid(Q)$ which acts on $D_{fd}(\Gamma)$. Keller and Nicol\'as prove that there is a canonical bijection between the set of $Braid(Q)$-orbits of reachable cluster collections and reachable cluster-tilting sequences in the cluster category associated to $D_{fd}(Q)$ \cite{140, 180}. A discrete central charge with finitely many stable objects induces reachable cluster collections. We can view the images of these cluster collections as 'stable' objects in the cluster category.\\

A cluster collection $S'$ is \textit{nearby} if the associated heart $\mathcal{A}'$ is the left-tilt of some torsion pair in $\mathcal{A}$. A sequence of simple tilts at objects of $\mathcal{A}$ starting at the initial cluster collection $S$ gives a sequence of nearby cluster collections 
\begin{align}
S=S^{0}\longrightarrow S^{1}\longrightarrow\cdots\longrightarrow S^{N}. \nonumber
\end{align}
 For a sequence of reachable nearby cluster collections given by a sequence of simple tilts B. Keller introduced in \cite{140} the invariant in $\mathbb{A}_{Q}$
\begin{align}
\label{invariant}
\mathbb{E}(\epsilon_{1}\beta_{1})^{\epsilon_{1}}\cdots\mathbb{E}(\epsilon_{N}\beta_{N})^{\epsilon_{N}}
\end{align}
where $\beta_{i},1\leq i \leq N$ is the class of the i-th simple object on that we tilt. If this object is an element of $\mathcal{A}$ we set $\epsilon_{i}=+1$, if it is an element of $\mathcal{A}[-1]$ we set $\epsilon_{i}=-1$.  
  
\begin{thm}\cite{100,140}
Let be given sequences of reachable nearby cluster collections as described above with the same final nearby cluster collection. Then the invariant \ref{invariant} does not depend on the choice of a sequence.  
\end{thm}

In our case we tilt at the stable objects of $\mathcal{A}$ in the order of decreasing phase.

\begin{prop}
Let $(Q,W)$ be a quiver with non-degenerate potential as in Theorem \ref{theorem}. Then the refined Donaldson-Thomas invariant $\mathbb{E}_{Q,W,Z}$ does not depend on the chosen discrete central charge $Z:K(\mathcal{A})\rightarrow\mathbb{C}$ with finitely many stable objects.
\end{prop} 

Note that the potential does not have to be polynomial. In the case of a Dynkin quiver this proves the identities of Reineke \cite{150}.

\section*{Acknowledgements} 
I thank Sergio Cecotti, Dmytro Shklyarov, Katrin Wendland, Dan Xie and in particular Bernhard Keller for helpful discussions or correspondences. This research was supported by the DFG-Graduiertenkolleg GRK 1821 "Cohomological Methods in Geometry" at the University of Freiburg. I thank the Simons Center for Geometry and Physics in Stony Brook for hospitality. Part of this work was carried out during a visit in Stony Brook.

\end{document}